\documentclass[a4paper,12pt]{amsart}
\title[Congruence spectra and geometric reconstruction of $\boldsymbol{T}$-algebras]{Congruence spectra and geometric reconstruction of $\boldsymbol{T}$-algebras}

\author{Song JuAe}
\address{Faculty of Mathematics, Kyushu University, 744 Motooka, Nishi-ku, Fukuoka, 819-0395, Japan.}
\email{songjuae@math.kyushu-u.ac.jp}


\usepackage{xcolor}
\definecolor{cadmiumgreen}{rgb}{0.0, 0.42, 0.24}
\usepackage{amsmath,amssymb}
\usepackage{amsthm}
\usepackage{mathrsfs}
\usepackage{appendix}
\usepackage{geometry}
\usepackage[
    colorlinks=true,
    linkcolor=blue,
    urlcolor=blue,
    citecolor=cadmiumgreen,
]{hyperref}

\theoremstyle{plain}
\newtheorem{theorem}{Theorem}[section]
\newtheorem{proposition}[theorem]{Proposition}

\newtheorem{corollary}[theorem]{Corollary}

\theoremstyle{definition}
\newtheorem{definition}[theorem]{Definition}
\newtheorem{remark}[theorem]{Remark}
\newtheorem{example}[theorem]{Example}

\begin{document}

\begin{abstract}
We develop a geometric theory of congruence spectra for semirings.
We define the congruence spectrum for a semiring as the set of prime congruences on it, equip it with the Zariski topology, and construct a structure sheaf.
For a class of semirings including $\boldsymbol{B}$-algebras, we establish quasi-compactness of their congruence spectra.

We then prove reconstruction results for congruence spectra of quotients $K/\boldsymbol{E}(V)$, where $V$ is a subset of Euclidean space, $K$ is a tropical Laurent polynomial semiring, a tropical Laurent polynomial function semiring, or a tropical rational function semifield, and $\boldsymbol{E}(V)$ is the congruence of pairs of elements of $K$ that agree on $V$.
The global sections of the structure sheaf recover $K/\boldsymbol{E}(V)$, and the corresponding algebraic and geometric categories are contravariantly equivalent.

For finite unions of $\boldsymbol{R}$-rational polyhedral sets $V$, we completely classify the prime congruences on $K / \boldsymbol{E}(V)$.
As an application, we study rational function semifields of tropical curves with parallel rays.
We show that their congruence spectra recover the underlying curves together with the parallelism of rays and their intrinsic metrics.
\end{abstract}

\maketitle

\section{Introduction}
	\label{section1}

\subsection{Background and motivation}
    \label{subsection1.1}

The construction of affine schemes from commutative rings is one of the fundamental principles of algebraic geometry.
Given a ring $A$, the spectrum $\operatorname{Spec}(A)$ together with its structure sheaf $\mathcal{O}_{\operatorname{Spec}(A)}$ provides a geometric object whose global sections recover the original algebra:
\begin{align*}
\Gamma(\operatorname{Spec}(A),\mathcal{O}_{\operatorname{Spec}(A)}) \cong A.
\end{align*}
This correspondence provides a bridge between algebra and geometry, allowing algebraic objects to be recovered from their geometric realizations.

In the absence of subtraction, congruences provide the appropriate algebraic replacement of ideals for semirings.
This suggests replacing the prime ideal spectrum by a space of prime congruences.
To obtain an analogue of affine schemes, however, we equip this space with a structure sheaf.
We show that, for the classes of semirings considered in abstract tropical algebraic geometry, the resulting space recovers the original semiring through its global sections and also recovers geometric information.

\subsection{Main results}
    \label{subsection1.2}

The main results of this paper are summarized as follows.

\begin{itemize}
\item We construct the congruence spectrum $\operatorname{CSpec}(S)$ together with the Zariski topology and its structure sheaf $\mathcal{O}_{\operatorname{CSpec}(S)}$ for a semiring $S$ (Subsection~\ref{subsection3.1} and Definition~\ref{dfn1}).
For a class including $\boldsymbol{B}$-algebras, we establish quasi-compactness of their congruence spectra (Subsection~\ref{subsection3.2}).
Here, $\boldsymbol{B}$ denotes the Boolean semifield.
\item For a subset $V$ of the $n$-dimensional Euclidean space $\boldsymbol{R}^n$, we prove that
\begin{align*}
\Gamma(\operatorname{CSpec}(K / \boldsymbol{E}(V)), \mathcal{O}_{\operatorname{CSpec}(K / \boldsymbol{E}(V))}) \cong K / \boldsymbol{E}(V),
\end{align*}
where $K$ is one of the following semirings and semifields in $n$ variables: the tropical Laurent polynomial semiring $\boldsymbol{T}[\boldsymbol{X}^{\pm}]_n$, the tropical Laurent polynomial function semiring $\overline{\boldsymbol{T}[\boldsymbol{X}^{\pm}]}_n$, and the tropical rational function semifield $\overline{\boldsymbol{T}(\boldsymbol{X})}_n$ (Theorem~\ref{thm1}).
Here, $\boldsymbol{E}(V)$ denotes the congruence of pairs of elements of $K$ that agree on $V$.
\item We establish a categorical equivalence between the category of $\boldsymbol{T}$-algebras isomorphic to some $\overline{\boldsymbol{T}(\boldsymbol{X})}_n / \boldsymbol{E}(V)$ with $\boldsymbol{T}$-algebra homomorphisms and the category of the pairs $(\operatorname{CSpec}(S), \mathcal{O}_{\operatorname{CSpec}(S)})$ with appropriate morphisms for such $\boldsymbol{T}$-algebras $S$ (Theorem~\ref{thm2}).
The same categorical equivalence also holds with $\overline{\boldsymbol{T}(\boldsymbol{X})}_n$ replaced by $\boldsymbol{T}[\boldsymbol{X}^{\pm}]_n$ or $\overline{\boldsymbol{T}[\boldsymbol{X}^{\pm}]}_n$.
\item We completely classify the prime congruences in $\operatorname{CSpec}(K / \boldsymbol{E}(V))$ when $V$ is a finite union of $\boldsymbol{R}$-rational polyhedral sets in terms of t-admissible matrices (Theorem~\ref{thm3}).
\item As a fundamental class of examples of $\overline{\boldsymbol{T}(\boldsymbol{X})}_n / \boldsymbol{E}(V)$, we study rational function semifields of tropical curves with parallel rays.
We show that their congruence spectra recover the underlying tropical curves together with the parallelism of rays and their intrinsic metrics (Subsection~\ref{subsection5.2}).
\end{itemize}

Together, these results show that the congruence spectrum equipped with its structure sheaf captures not only the algebraic structure of the original semiring, but also geometric information in the tropical setting.

\subsection{Comparison with previous works}
    \label{subsection1.3}

There have been several approaches to developing algebraic foundations for tropical geometry using semirings and related algebraic structures.

One line of development concerns scheme-theoretic approaches to tropical geometry.
Giansiracusa and Giansiracusa introduced bend relations and bend congruences to define scheme-theoretic structures associated with tropicalizations (\cite{Giansiracusa=Giansiracusa2}).
Their approach provides a refinement of the underlying tropical variety by retaining congruence-theoretic information.
This framework was further developed in several directions.
For example, Maclagan and Rinc\'on developed the theory of tropical ideals and studied the associated tropical scheme structures, as well as their relations with tropical cycles and valuated matroids (\cite{Maclagan=Rincon1,Maclagan=Rincon2}).

Another line of research focuses on prime congruences on idempotent semirings.
Jo\'{o} and Mincheva introduced and studied prime congruences and the corresponding notion of Krull dimension for $\boldsymbol{B}$-algebras (\cite{Joo=Mincheva}).
These works have yielded explicit results on the Krull dimensions of quotient $\boldsymbol{T}$-algebras of tropical polynomial and Laurent polynomial semirings and tropical rational function semifields by several types of congruences (\cite{Joo=Mincheva3, Nakajima=Song}).

More recently, Friedenberg and Mincheva gave a geometric classification of prime congruences containing a given congruence with a finite tropical basis on a toric semiring (\cite{Friedenberg=Mincheva}).
Their results provide strong characterizations of such prime congruences and establish, among other applications, a connection between the bend-congruence approach to tropical geometry and tropical function semirings.
For a more detailed comparison between their classification and the classification given in Theorem~\ref{thm3}, see Remark~\ref{rem4}.
Tanaka also developed a framework for studying spaces of prime congruences on tropical algebras using flags, and constructed global spaces by gluing such prime congruence spaces in analogy with toric geometry (\cite{Tanaka}). 

While spaces of prime congruences and their topologies have been studied in the previous works mentioned above, our primary object is the congruence spectrum equipped with a structure sheaf.
This allows us to study the relation between the algebraic structure of a semiring and the local and global geometry of its congruence spectrum in a manner analogous to affine geometry.

Note also that the scheme-theoretic approach of Giansiracusa and Giansiracusa cannot reconstruct tropical curves with parallel rays from their rational function semifields, since semifields have only trivial ideals.
Our congruence spectra give an algebraic, direct reconstruction of such curves from semifields.

\subsection{Organization}
    \label{subsection1.4}

The paper is organized as follows.
In Section~\ref{section2}, we recall the basic definitions and results used throughout the paper.
In Section~\ref{section3}, we introduce the congruence spectrum, equip it with the Zariski topology and a structure sheaf, and establish the categorical equivalences.
In Section~\ref{section4}, we study the case where the underlying subset is a finite union of $\boldsymbol{R}$-rational polyhedral sets and classify the prime congruences.
In Section~\ref{section5}, we apply the results to tropical curves with parallel rays and investigate how their geometry is reflected in their congruence spectra.

\section*{Acknowledgements}
The author thanks Yasuhito Nakajima, Kentaro Tanaka, Takaaki Ito, Masanori Kobayashi and Yuki Tsutsui for their helpful comments.
This work was supported by JSPS KAKENHI Grant Number 25K17230.
During the preparation of this manuscript, the author used generative AI solely for English language editing and abstract structuring. 
The mathematical content and results remain entirely the work of the author.

\section{Preliminaries}
	\label{section2}

In this section, we recall the basic notions and results on semirings, congruences, and tropical rational function semifields that will be used throughout the paper.

\subsection{Semirings, algebras and semifields}
	\label{subsection2.1}

In this paper, a \textit{semiring} $S$ is a commutative semiring with the absorbing identity $0_S$ for addition $+$ and the identity $1_S$ for multiplication $\cdot$, i.e., a nonempty set $S$ with two binary operations $+$ and $\cdot$ satisfying the following four conditions:

$(1)$ the pair $(S, +)$ is a commutative monoid with an identity $0_S$,

$(2)$ the pair $(S, \cdot)$ is a commutative monoid with an identity $1_S$,

$(3)$ the multiplication $\cdot$ distributes over the addition $+$, i.e., $(a + b) \cdot c = a \cdot c + b \cdot c$ holds for any $a,b, c \in S$, and

$(4)$ the zero element $0_S$ is absorbing, i.e., $0_S \cdot a = 0_S$ holds for any $a \in S$.

If every nonzero element of a semiring $S$ is multiplicatively invertible and $0_S \not= 1_S$, then $S$ is called a \textit{semifield}.

The set $\boldsymbol{T} := \boldsymbol{R} \cup \{ -\infty \}$ equipped with the tropical operations:
\begin{align*}
a \oplus b := \operatorname{max}\{ a, b \} \quad	\text{and} \quad a \odot b := a + b,
\end{align*}
where $a, b \in \boldsymbol{T}$ and $a + b$ stands for the usual sum of $a$ and $b$, becomes a semifield.
Here, for any $a \in \boldsymbol{T}$, we handle $-\infty$ as follows:
\begin{align*}
a \oplus (-\infty) = (-\infty) \oplus a = a \quad \text{and} \quad a \odot (-\infty) = (-\infty) \odot a = -\infty.
\end{align*}
This triple $(\boldsymbol{T}, \oplus, \odot)$ is called the \textit{tropical semifield}.
The subset $\boldsymbol{B} := \{ 0, -\infty \}$ of $\boldsymbol{T}$ becomes a semifield with tropical operations of $\boldsymbol{T}$ and is called the \textit{Boolean semifield}.

The \textit{tropical Laurent polynomials} are defined as Laurent polynomials with respect to tropical operations $\oplus$ and $\odot$, where the coefficients are taken in $\boldsymbol{T}$.
The set of all tropical Laurent polynomials in $n$ variables is denoted by $\boldsymbol{T}[X_1^{\pm}, \ldots, X_n^{\pm}] = \boldsymbol{T}[\boldsymbol{X}^{\pm}]_n$.
It becomes a semiring with two tropical operations and is called the \textit{tropical Laurent polynomial semiring}.
Throughout this paper, our usage of the symbol $a \odot \boldsymbol{X}^{\odot \boldsymbol{m}}$ with $a \in \boldsymbol{T}$ and $\boldsymbol{m} = (m_1, \ldots, m_n) \in \boldsymbol{Z}^n$ means the tropical monomial $a \odot X_1^{\odot m_1} \odot \cdots \odot X_n^{\odot m_n}$.

A map $\varphi : S_1 \to S_2$ between semirings is a \textit{semiring homomorphism} if for any $a, b \in S_1$,
\begin{align*}
\varphi(a + b) = \varphi(a) + \varphi(b), \	\varphi(a \cdot b) = \varphi(a) \cdot \varphi(b), \	\varphi(0_{S_1}) = 0_{S_2}, \	\text{and}\	\varphi(1_{S_1}) = 1_{S_2}.
\end{align*}
When $\varphi$ is a bijective semiring homomorphism, it is a \textit{semiring isomorphism}.
Then $S_1$ is \textit{isomorphic to} $S_2$ as a semiring. 
Given a semiring homomorphism $\varphi:S_1\to S_2$, the pair $(S_2,\varphi)$ is called an $S_1$-algebra; when the structure homomorphism is understood, we simply write $S_2$.
In particular, if both $S_1$ and $S_2$ are semifields and $\varphi$ is injective, then $S_2$ is a \textit{semifield over $S_1$}.
Then $S_1$ is frequently identified with its image by $\varphi$ in $S_2$.
For a semiring $S_1$, a map $\psi : (S_2, \varphi) \to (S_2^{\prime}, \varphi^{\prime})$ between $S_1$-algebras is a \textit{$S_1$-algebra homomorphism} if $\psi$ is a semiring homomorphism and $\varphi^{\prime} = \psi \circ \varphi$.
When there is no confusion, we write $\psi : S_2 \to S_2^{\prime}$ simply.
A bijective $S_1$-algebra homomorphism $S_2 \to S_2^{\prime}$ is a \textit{$S_1$-algebra isomorphism}.
Then $S_2$ and $S_2^{\prime}$ are said to be \textit{isomorphic} and denoted by $S_2 \cong S_2^{\prime}$.

A semiring $S$ is a $\boldsymbol{B}$-algebra if and only if it is \textit{(additively) idempotent}, i.e., $a + a = a$ holds for any $a \in S$.
Both $\boldsymbol{T}$ and $\boldsymbol{T}[\boldsymbol{X}^{\pm}]_n$ are $\boldsymbol{B}$-algebras.
A $\boldsymbol{B}$-algebra $S$ is said to be \textit{cancellative} if whenever $a \cdot b = a \cdot c$ for some $a, b, c \in S$, then either $a = 0_S$ or $b = c$.
If $S$ is cancellative, its semifield of fractions $Q(S)$ can be constructed analogously to the field of fractions of an integral domain.
In this case, the map $S \to Q(S); a \mapsto a/1_S$ becomes an injective $\boldsymbol{B}$-algebra homomorphism.

\subsection{Congruences}
	\label{subsection2.2}

A \textit{congruence} $C$ on a semiring $S$ is a subset of $S^2 = S \times S$ satisfying

$(1)$ $(a, a) \in C$ holds for any $a \in S$,

$(2)$ if $(a, b) \in C$, then $(b, a) \in C$,

$(3)$ if $(a, b) \in C$ and $(b, c) \in C$, then $(a, c) \in C$,

$(4)$ if $(a, b) \in C$ and $(c, d) \in C$, then $(a + c, b + d) \in C$, and

$(5)$ if $(a, b) \in C$ and $(c, d) \in C$, then $(a \cdot c, b \cdot d) \in C$.

The set $S^2$ is a congruence on $S$ called the \textit{improper} congruence on $S$.
Congruences other than the improper congruence are said to be \textit{proper}.
Quotients by congruences can be considered in the usual sense and the quotient semiring of $S$ by the congruence $C$ is denoted by $S / C$.
Then the natural surjection $\pi_C : S \twoheadrightarrow S / C$ is a semiring homomorphism.

The intersection of (possibly infinitely many) congruences is again a congruence.
For a subset $T$ of $S^2$, let $\langle T \rangle_S$ be the smallest congruence on $S$ containing $T$, i.e., the intersection of all congruences on $S$ containing $T$.

For a semiring homomorphism $\psi : S_1 \to S_2$, the \textit{kernel congruence} $\operatorname{Ker}(\psi)$ of $\psi$ is the congruence $\{ (a, b) \in S_1^2 \,|\, \psi(a) = \psi(b) \}$ on $S_1$.
Also, for a congruence $C$ on $S_2$, the set $\psi^{-1}(C) := \{ (a, b) \in S_1^2 \,|\, (\psi(a), \psi(b)) \in C \}$ is a congruence on $S_1$.
For two congruences $E, F$ on a semiring $S$ such that $E \subset F$, the image $F/E$ of $F$ under the natural map $S^2\to(S/E)^2$ is a congruence on $S / E$.

\subsection{Prime congruences and Krull dimensions}
	\label{subsection2.3}

A congruence $C$ on a semiring $S$ is \textit{prime} if $C$ is proper and $(a \cdot c + b \cdot d, a \cdot d + b \cdot c) \in C$ implies that $(a, b) \in C$ or $(c, d) \in C$ holds.
This definition was given in \cite{Joo=Mincheva}.

For a semiring homomorphism $\psi : S_1 \to S_2$ and a prime
congruence $P$ on $S_2$, the congruence $\psi^{-1}(P)$ on $S_1$ is prime.

We define the \textit{Krull dimension} $\operatorname{dim}S$ of $S$ as the supremum of all
integers $l \ge 0$ such that there exists a chain $P_0 \subsetneq P_1 \subsetneq \cdots \subsetneq P_l$ of prime congruences on $S$.

Let $U$ be an $l \times (n + 1)$-matrix with entries in $\boldsymbol{R}$ with an integer $l$ such that $1 \le l \le n + 1$.
This matrix $U$ is \textit{t-admissible} if

$(1)$ its first column $\boldsymbol{u}_1$ is the zero vector $\boldsymbol{0} \in \boldsymbol{R}^l$ or the first nonzero entry of $\boldsymbol{u}_1$ is positive, and 

$(2)$ for any integer $i$ such that $1 \le i \le l$, its $i$th row is \textit{irredundant} on $\boldsymbol{R} \times \boldsymbol{Z}^n$, i.e., there exist $a \in \boldsymbol{R}$ and $\boldsymbol{m} \in \boldsymbol{Z}^n$ such that the first nonzero entry of $U \begin{pmatrix} a \\ \boldsymbol{m} \end{pmatrix}$ is the $i$th entry.

For convenience, we also regard $U(0) := \begin{pmatrix} 0 & \cdots  & 0\end{pmatrix}$ as t-admissible throughout this paper.

When $U$ above is t-admissible, then by \cite[Proposition~2.10(ii)]{Joo=Mincheva}, $U$ defines a prime congruence $P(U)$ on $\boldsymbol{T}[\boldsymbol{X}^{\pm}]_n$ generated by the set of pairs
\begin{align*}
\left(a_1 \odot \boldsymbol{X}^{\odot \boldsymbol{m}_1} \oplus a_2 \odot \boldsymbol{X}^{\odot \boldsymbol{m}_2}, a_1 \odot \boldsymbol{X}^{\odot \boldsymbol{m}_1} \right)
\end{align*}
such that $U \begin{pmatrix} a_1 - a_2 \\ \boldsymbol{m}_1 - \boldsymbol{m}_2 \end{pmatrix}$ is $\boldsymbol{0}$ or its first nonzero entry is positive.
Note that $P(U)$ is written as $P(U)_{\boldsymbol{T}}$ in \cite{Joo=Mincheva}.
For $1 \le i \le l$, let $U(i)$ denote the $i \times (n + 1)$-matrix consisting of the first $i$ rows of $U$.
By definition, $U(i)$ is t-admissible for any $1 \le i \le l$, and
\begin{align*}
    P(U) = P(U(l)) \subsetneq P(U(l - 1)) \subsetneq \cdots \subsetneq P(U(1)) \subsetneq P(U(0)).
\end{align*}
By \cite[Theorem~4.13(iii)]{Joo=Mincheva}, every proper congruence on $\boldsymbol{T}[\boldsymbol{X}^{\pm}]_n$ containing $P(U)$ is of the form $P(U(i))$ for some $0 \le i \le l$.
In particular, $\operatorname{dim}\boldsymbol{T}[\boldsymbol{X}^{\pm}]_n / P(U(i)) = i$.
Moreover, by \cite[Theorem~4.14(i)]{Joo=Mincheva}, for every prime congruence $P$ on $\boldsymbol{T}[\boldsymbol{X}^{\pm}]_n$, there exists a t-admissible matrix $U$ such that $P = P(U)$.

\subsection{Tropical rational function semifields and congruence varieties}
	\label{subsection2.4}

Let $\overline{C}$ be the set $    \left\{ (f, g) \in \boldsymbol{T}[\boldsymbol{X}^{\pm}]_n^2 \,\middle|\, \forall x \in \boldsymbol{R}^n, f(x) = g(x) \right\}$.
Then $\overline{C}$ is a congruence on $\boldsymbol{T}[\boldsymbol{X}^{\pm}]_n$ and the semiring $\overline{\boldsymbol{T}[X_1^{\pm}, \ldots, X_n^{\pm}]} = \overline{\boldsymbol{T}[\boldsymbol{X}^{\pm}]}_n := \boldsymbol{T}[\boldsymbol{X}^{\pm}]_n / \overline{C}$ is a cancellative $\boldsymbol{B}$-algebra by \cite[Lemma~4.3]{Nakajima=Song}.
We call it the \textit{tropical Laurent polynomial function semiring}.
We call its semifield of fractions the \textit{tropical rational function semifield} and write it as $\overline{\boldsymbol{T}(X_1, \ldots, X_n)} = \overline{\boldsymbol{T}(\boldsymbol{X})}_n$.
In what follows, by abuse of notation, the image of each $X_i$ in $\overline{\boldsymbol{T}(\boldsymbol{X})}_n$ is again written by $X_i$.
Let $K$ be $\boldsymbol{T}[\boldsymbol{X}^{\pm}]_n, \overline{\boldsymbol{T}[\boldsymbol{X}^{\pm}]}_n$ or $\overline{\boldsymbol{T}(\boldsymbol{X})}_n$.
For a subset $T$ of $K^2$, we define $\boldsymbol{V}(T) := \{ x \in \boldsymbol{R}^n \,|\, \forall (f, g) \in T, f(x) = g(x)\}$.
Since $\boldsymbol{V}(T) = \boldsymbol{V}(\langle T \rangle_K)$, we call it the \textit{congruence variety} associated with $T$.
Declaring the congruence varieties to be the closed subsets of
$\boldsymbol{R}^n$ defines a topology on $\boldsymbol{R}^n$,
which coincides with the Euclidean topology by
\cite[Proposition~3.8]{JuAe5} (and by the same argument for the cases $K = \boldsymbol{T}[\boldsymbol{X}^{\pm}]_n, \overline{\boldsymbol{T}[\boldsymbol{X}^{\pm}]}_n$).
For a subset $V$ of $\boldsymbol{R}^n$, we write
\begin{align*}
\boldsymbol{E}(V) := \left\{ (f, g) \in K^2 \,\middle|\, \forall x \in V, f(x) = g(x) \right\}.
\end{align*}
Then $\boldsymbol{E}(V)$ is a congruence on $K$.
Note that $\overline{C} = \boldsymbol{E}(\boldsymbol{R}^n)$ for $K = \boldsymbol{T}[\boldsymbol{X}^{\pm}]_n$.

\section{Congruence spectrum}
    \label{section3}

In this section, we develop the basic geometric framework associated with the congruence spectrum of a semiring.
We introduce the Zariski topology, establish quasi-compactness for a class including $\boldsymbol{B}$-algebras, construct the structure sheaf, and study the resulting functorial and categorical structures.

\subsection{Congruence spectrum and Zariski topology}
    \label{subsection3.1}

For a semiring $S$, we define the \textit{congruence spectrum} of $S$, denoted by $\operatorname{CSpec}(S)$, to be the set of prime congruences on $S$.
For a congruence $C$ on $S$, define
\begin{align*}
V(C) := \{ P \in \operatorname{CSpec}(S) \,|\, C \subset P \}.    
\end{align*}
By \cite[Proposition 2.4]{Tanaka}, the sets $V(C)$ satisfy the axioms of a closed set system.
Hence they define the \textit{Zariski topology} on $\operatorname{CSpec}(S)$.

\begin{example}
    \label{ex1}
It is well-known that for a $\boldsymbol{B}$-algebra $S$ consisting of at least two elements, $\operatorname{CSpec}(S)$ has at least one element.
In fact, since $0_S \not= 1_S$ in this case, by \cite[Proposition~2.5]{Joo=Mincheva} and Zorn's lemma, there is a surjective $\boldsymbol{B}$-algebra homomorphism $S \twoheadrightarrow \boldsymbol{B}$, and its kernel congruence is prime by \cite[Proposition~2.10(ii)]{Joo=Mincheva}.
See \cite[the paragraph immediately preceding Proposition~2.5]{Joo=Mincheva}.
Also, this means that every proper congruence $C$ on $S$ is contained in the kernel congruence of $S / C \twoheadrightarrow \boldsymbol{B}$ by \cite[Proposition~2.13]{Joo=Mincheva2}.
In particular, if $S = \overline{\boldsymbol{T}(\boldsymbol{X})}_n$, then by \cite[Lemma~2.2]{JuAe5} and \cite[(9.10)EXAMPLE]{Golan}, such a prime congruence $P$ is unique and the whole space $\operatorname{CSpec}(\overline{\boldsymbol{T}(\boldsymbol{X})}_n)$ is the unique open subset of $\operatorname{CSpec}(\overline{\boldsymbol{T}(\boldsymbol{X})}_n)$ containing $P$.
Since $\operatorname{CSpec}(\overline{\boldsymbol{T}(\boldsymbol{X})}_n)$ has at least two elements, for example, the kernel congruences of the surjective $\boldsymbol{T}$-algebra homomorphisms $\overline{\boldsymbol{T}(\boldsymbol{X})}_n \twoheadrightarrow \boldsymbol{T}; f \mapsto f(x)$ for arbitrary $x \in \boldsymbol{R}^n$ and $\overline{\boldsymbol{T}(\boldsymbol{X})}_n \twoheadrightarrow \boldsymbol{B}; f \not= -\infty \mapsto 0, -\infty \mapsto -\infty$, both of which are prime by \cite[Proposition~2.10(ii)]{Joo=Mincheva}, $\operatorname{CSpec}(\overline{\boldsymbol{T}(\boldsymbol{X})}_n)$ is not Hausdorff.
\end{example}

\subsection{Quasi-compactness}
    \label{subsection3.2}
    
\begin{proposition}
    \label{prop1}
For a semiring $S$, if $S$ satisfies condition $(\ast)$ below, then the topological space $\operatorname{CSpec}(S)$ is quasi-compact.

$(\ast)$ Every proper congruence on $S$ is contained in some prime congruence on $S$.
\end{proposition}

\begin{proof}
If $\operatorname{CSpec}(S)$ is empty, then it is quasi-compact.

Assume that $\operatorname{CSpec}(S)$ is nonempty.
Let $\{F_i\}_{i \in I}$ be a family of closed subsets of $\operatorname{CSpec}(S)$ with the finite intersection property. 
For each $i \in I$, let $C_i$ be a congruence on $S$ such that $F_i = V(C_i)$. 
Let $C := \left\langle \bigcup_{i \in I} C_{i} \right\rangle_S$.

If $C = S^2$, then by \cite[Lemma~2.1]{Nakajima=Song}, there exist finitely many indices $i_1, \dots, i_m \in I$ such that $\left\langle \bigcup_{j=1}^m C_{i_j} \right\rangle_S \ni (1_S, 0_S)$.
Hence $\left\langle \bigcup_{j=1}^m C_{i_j} \right\rangle_S = S^2$ by \cite[Lemma~3.3]{JuAe6}.
However, this makes a contradiction such that $\varnothing = V (S^2) = V\left( \left\langle \bigcup_{j=1}^m C_{i_j} \right\rangle_S \right) = \bigcap_{j = 1}^m V(C_{i_j}) = \bigcap_{j = 1}^m F_{i_j} \not= \varnothing$, which follows from the finite intersection property.
Therefore, $C \not= S^2$, meaning that $C$ is a proper congruence.

By $(\ast)$, there exists a prime congruence $P \in \operatorname{CSpec}(S)$ such that $C_{i} \subset C \subset P$ holds for each $i \in I$, and so $P \in V(C_{i}) = F_i$.
Consequently, $P \in \bigcap_{i \in I} F_i$, which implies that $\operatorname{CSpec}(S)$ is quasi-compact.
\end{proof}

By Example~\ref{ex1}, Proposition~\ref{prop1} and \cite[Proposition~2.8]{Tanaka}, we have the following:

\begin{corollary}
    \label{cor1}
For a $\boldsymbol{B}$-algebra $S$, the topological space $\operatorname{CSpec}(S)$ is quasi-compact.
\end{corollary}

\begin{corollary}
    \label{cor2}
Let $V$ be a subset of $\boldsymbol{R}^n$ and $K = \boldsymbol{T}[\boldsymbol{X}^{\pm}]_n, \overline{\boldsymbol{T}[\boldsymbol{X}^{\pm}]}_n, \overline{\boldsymbol{T}(\boldsymbol{X})}_n$.
Then $V(\boldsymbol{E}(V))$ is a quasi-compact subset of $\operatorname{CSpec}(K)$ and $\operatorname{CSpec}(K / \boldsymbol{E}(V))$ is quasi-compact.
\end{corollary}

\subsection{Structure sheaf}
    \label{subsection3.3}

We construct the structure sheaf on the congruence spectrum and study its global sections.

\begin{definition}[Structure Sheaf]
    \label{dfn1}
Let $S$ be a semiring.
For an open subset $U \subset \operatorname{CSpec}(S)$, let 
\begin{align*}
&\mathcal{O}_{\operatorname{CSpec}(S)}(U) \\
:=& \left\{ s : U \to \coprod_{P \in U} S/P \,\middle|\, \begin{array}{l} \forall P \in U, \, s(P) \in S/P \text{ and } \\ s \text{ is locally represented by an element of } S \text{ at } P \end{array} \right\}.  
\end{align*}
Here, $s$ is \textit{locally represented} by $f \in S$ at $P$ if there exists a neighborhood $V \subset U$ of $P$ such that for any $Q \in V$, $s(Q) = [f] \in S / Q$.
We call an element of $\mathcal{O}_{\operatorname{CSpec}(S)}(U)$ a \textit{section}.
In particular, we write $\mathcal{O}_{\operatorname{CSpec}(S)}(\operatorname{CSpec}(S))$ as $\Gamma(\operatorname{CSpec}(S), \mathcal{O}_{\operatorname{CSpec}(S)})$ and call its elements \textit{global sections}.
It is straightforward to verify that the assignment $U \mapsto \mathcal{O}_{\operatorname{CSpec}(S)}(U)$ with the natural restriction maps forms a sheaf of semirings on $\operatorname{CSpec}(S)$. 
We call it the \textit{structure sheaf} on $\operatorname{CSpec}(S)$.

We call the following semiring homomorphism the \textit{canonical map}:
\begin{align*}
    \eta_S : S \to \Gamma(\operatorname{CSpec}(S), \mathcal{O}_{\operatorname{CSpec}(S)}); f \mapsto s_f,
\end{align*}
where $s_f(P) := [f] \in S/P$ for any $P \in \operatorname{CSpec}(S)$.
\end{definition}

\begin{remark}
    \label{rem1}
Let $\pi : \boldsymbol{T}[\boldsymbol{X}^{\pm}]_n \twoheadrightarrow \overline{\boldsymbol{T}[\boldsymbol{X}^{\pm}]}_n = \boldsymbol{T}[\boldsymbol{X}^{\pm}]_n / \boldsymbol{E}(\boldsymbol{R}^n)$ be the natural surjective $\boldsymbol{T}$-algebra homomorphism.
Then, $\operatorname{Ker}(\pi) = \boldsymbol{E}(\boldsymbol{R}^n)$.
Moreover, each prime congruence on $\boldsymbol{T}[\boldsymbol{X}^{\pm}]_n$ contains $\boldsymbol{E}(\boldsymbol{R}^n)$.
In fact, by \cite[Theorem~5.4(i)]{Joo=Mincheva}, it is enough to see that $\boldsymbol{E}(\boldsymbol{R}^n) \subset \bigcap_{U : \text{a t-admissible matrix such that} \operatorname{dim}\boldsymbol{T}[\boldsymbol{X}^{\pm}]_n / P(U) \le 1}P(U)$, and this follows immediately from the definition of $P(U)$ for such $U$.
By \cite[Proposition~2.13]{Joo=Mincheva2}, $\pi$ induces a containment-preserving homeomorphism between $\operatorname{CSpec}(\boldsymbol{T}[\boldsymbol{X}^{\pm}]_n)$ and $\operatorname{CSpec}(\overline{\boldsymbol{T}[\boldsymbol{X}^{\pm}]}_n)$.
Let $\iota: \overline{\boldsymbol{T}[\boldsymbol{X}^{\pm}]}_n \hookrightarrow \overline{\boldsymbol{T}(\boldsymbol{X})}_n; f \mapsto f \odot 0^{\odot (-1)} = f$ be the natural injective $\boldsymbol{T}$-algebra homomorphism.
By \cite[Proposition~3.9(iii)]{Joo=Mincheva2}, $\iota$ induces a containment-preserving bijection between $\operatorname{CSpec}(\overline{\boldsymbol{T}[\boldsymbol{X}^{\pm}]}_n)$ and $\operatorname{CSpec}(\overline{\boldsymbol{T}(\boldsymbol{X})}_n)$.
This bijection preserves the Zariski closed sets by \cite[Proposition~3.9(i)--(iii)]{Joo=Mincheva2}, and hence is a homeomorphism.
Consequently, all the topological spaces 
\begin{align*}
\operatorname{CSpec}(\boldsymbol{T}[\boldsymbol{X}^{\pm}]_n), \quad \operatorname{CSpec}(\overline{\boldsymbol{T}[\boldsymbol{X}^{\pm}]}_n), \quad \text{and} \quad \operatorname{CSpec}(\overline{\boldsymbol{T}(\boldsymbol{X})}_n)
\end{align*}
are homeomorphic to each other via the maps induced by $\pi$, $\iota$, and their composition $\iota\circ\pi$.
By abuse of notation, we use the same symbol $P(U)$ for the corresponding prime congruence on $\overline{\boldsymbol{T}[\boldsymbol{X}^{\pm}]}_n$ and on $\overline{\boldsymbol{T}(\boldsymbol{X})}_n$.
\end{remark}

For $V \subset \boldsymbol{R}^n$, let $K$ be $\boldsymbol{T}[\boldsymbol{X}^{\pm}]_n / \boldsymbol{E}(V), \overline{\boldsymbol{T}[\boldsymbol{X}^{\pm}]}_n / \boldsymbol{E}(V)$ or $\overline{\boldsymbol{T}(\boldsymbol{X})}_n / \boldsymbol{E}(V)$.

\begin{remark}
    \label{rem2}
The closure $\overline{V}$ of $V$ can be regarded as a subset of $\operatorname{CSpec}(K)$.
More precisely, the map $\varphi : \overline{V} \to \operatorname{CSpec}(K); x \mapsto \boldsymbol{E}(\{x\}) / \boldsymbol{E}(V)$ is a homeomorphism onto its image.
In fact, by definition, $\boldsymbol{E}(\{x\}) / \boldsymbol{E}(V)$ is a prime congruence on $K$ and $\varphi$ is injective.
Also, let $U$ be an open subset of $\operatorname{CSpec}(K)$.
Then there exists a congruence $C$ on $K$ such that $U = \operatorname{CSpec}(K) \setminus V(C)$.
The inverse image $\varphi^{-1}(U)$ is $\overline{V} \setminus \boldsymbol{V}(C)$.
Conversely, if $O$ is an open subset of $\overline{V}$, then $\varphi(O)$ is an open subset of $\varphi(\overline{V})$.
\end{remark}

\begin{theorem}[Global Sections]
    \label{thm1}
The canonical map $\eta_K : K \to \Gamma(\operatorname{CSpec}(K), \mathcal{O}_{\operatorname{CSpec}(K)})$ is a $\boldsymbol{T}$-algebra isomorphism.
\end{theorem}

\begin{proof}
If $V = \varnothing$, then $K = \{ -\infty \}$ and $\Gamma(\operatorname{CSpec}(\{-\infty\}), \mathcal{O}_{\operatorname{CSpec}(\{-\infty\})})$ is a semiring consisting of only one element.
Thus $\eta_{\{-\infty\}}$ is a $\boldsymbol{T}$-algebra isomorphism in this case.

Assume that $V \not= \varnothing$.
Then $K$ has at least two elements.

To show that $\eta_K$ is injective, suppose $s_f = s_g$ for $f, g \in K$. 
Then, $[f] = [g]$ in $K/P$ for all $P \in \operatorname{CSpec}(K)$. 
By Remark~\ref{rem2}, it follows that $f = g$ in $K$.
Thus, $\eta_K$ is injective.

To show that $\eta_K$ is surjective, let $s$ be an arbitrary global section.
By definition, there exists an open covering $\{U_i\}_{i \in I}$ of $\operatorname{CSpec}(K)$ such that on each $U_i$, $s$ is represented by some $f_i \in K$. 
By Example~\ref{ex1} and Remark~\ref{rem1}, there exists $j \in I$ such that $U_j = \operatorname{CSpec}(K)$.
Hence $s = s_{f_j}$. 
Therefore, $\eta_K$ is surjective, completing the proof.
\end{proof}

\subsection{Functors and categorical equivalences}
    \label{subsection3.4}

In this subsection, we study the functorial and categorical structures arising from subsets of Euclidean spaces and their associated $\boldsymbol{T}$-algebras.

Let $V$ (resp.~$W$) be a subset of $\boldsymbol{R}^n$ (resp.~$\boldsymbol{R}^m$).
A map $\theta : W \to V$ is a \textit{rational function map} if there exist $f_1, \ldots, f_n \in \overline{\boldsymbol{T}(\boldsymbol{Y})}_m$ such that $\theta(y) = (f_1(y), \ldots, f_n(y))$ for any $y \in W$.
In particular, there exists a unique rational function map $\varnothing \to V$, and for $V \not= \varnothing$, there exists no rational function map $V \to \varnothing$.

We define the category $\mathsf{Geom}_{\mathrm{concrete}}$ as follows.

\textbf{(i) Objects.} An object is a subset $V \subset \boldsymbol{R}^n$, with the ambient space $\boldsymbol{R}^n$ regarded as part of the data.

\textbf{(ii) Morphisms.} A morphism $W \to V$ is a rational function map $W \to V$.

\textbf{(iii) Composition.} Morphisms are composed by the usual composition of maps.

The identity map is a rational function map, and the composition of two rational function maps is again a rational function map.
Hence $\mathsf{Geom}_{\mathrm{concrete}}$ is a category.

Let $\mathsf{Geom}_{\mathrm{closed}}$ be the full subcategory of $\mathsf{Geom}_{\mathrm{concrete}}$ consisting of closed subsets of Euclidean spaces.

\smallskip

We define the category $\mathsf{Alg}_{\mathrm{presented}}$ as follows.

\textbf{(i) Objects.} An object is a presentation $\overline{\boldsymbol{T}(\boldsymbol{X})}_n / E$, where both $n$ and the congruence $E$ are part of the data.
In particular, two isomorphic $\boldsymbol{T}$-algebras may define distinct objects of $\mathsf{Alg}_{\mathrm{presented}}$ if they are given by different presentations.

\textbf{(ii) Morphisms.} A morphism from $\overline{\boldsymbol{T}(\boldsymbol{X})}_n / E \to \overline{\boldsymbol{T}(\boldsymbol{Y})}_m / F$ is a $\boldsymbol{T}$-algebra homomorphism $\overline{\boldsymbol{T}(\boldsymbol{X})}_n / E \to \overline{\boldsymbol{T}(\boldsymbol{Y})}_m / F$.

\textbf{(iii) Composition.} Morphisms are composed by the usual composition of maps.

Let $\mathsf{Alg}_{\boldsymbol{E}(V)\text{-type}}$ denote the full subcategory of $\mathsf{Alg}_{\mathrm{presented}}$ consisting of objects of the form
\begin{align*}
K_V = \overline{\boldsymbol{T}(\boldsymbol{X})}_n / \boldsymbol{E}(V).
\end{align*}

\smallskip

The results of \cite[Theorem 3.14, Lemma 3.18, Corollary 3.19, and Proposition 3.20]{JuAe5} can be summarized categorically as follows:
they give a contravariant equivalence
\begin{align*}
\mathsf{Geom}_{\mathrm{closed}}^{\mathrm{op}} \simeq \mathsf{Alg}_{\boldsymbol{E}(V)\text{-type}},
\end{align*}
where the functor in one direction is given on objects and on morphisms by
\begin{align*}
V \mapsto K_V \quad \text{and} \quad \theta \mapsto ( [f] \mapsto [f \circ \theta] ),
\end{align*}
while the functor in the other direction is given on objects and on morphisms by
\begin{align*}
\overline{\boldsymbol{T}(\boldsymbol{X})}_n / E \mapsto \boldsymbol{V}(E) \quad \text{and} \quad \psi \mapsto \theta_{\psi},
\end{align*}
where $\theta_{\psi}$ is the rational function map induced by the images of the coordinate functions.

\smallskip

Let $\mathsf{Alg}_{\mathrm{abstract}}$ be the category whose objects are $\boldsymbol{T}$-algebras $A$ for which there exist $n \ge 1$ and a surjective $\boldsymbol{T}$-algebra homomorphism $\overline{\boldsymbol{T}(\boldsymbol{X})}_n \twoheadrightarrow A$, and whose morphisms are $\boldsymbol{T}$-algebra homomorphisms.
Unlike $\mathsf{Alg}_{\mathrm{presented}}$, the category $\mathsf{Alg}_{\mathrm{abstract}}$ does not retain a choice of presentation as part of the data.

There is a natural forgetful functor
\begin{align*}
\mathcal{F} : \mathsf{Alg}_{\mathrm{presented}} \to \mathsf{Alg}_{\mathrm{abstract}},
\end{align*}
which sends a presented $\boldsymbol{T}$-algebra $\overline{\boldsymbol{T}(\boldsymbol{X})}_n/E$ to the underlying $\boldsymbol{T}$-algebra $\overline{\boldsymbol{T}(\boldsymbol{X})}_n/E$, forgetting the choice of its presentation.
On morphisms, $\mathcal{F}$ acts as the identity on the underlying $\boldsymbol{T}$-algebra homomorphisms.

Let $\mathsf{Alg}_{\mathrm{geom}}$ be the full subcategory $\mathcal{F}(\mathsf{Alg}_{\boldsymbol{E}(V)\text{-type}})$ of $\mathsf{Alg}_{\mathrm{abstract}}$.

\smallskip

We define the category $\mathsf{CSpec}$ as follows.

\textbf{(i) Objects.} An object is a pair $(\operatorname{CSpec}(S), \mathcal{O}_{\operatorname{CSpec}(S)} )$, where $S$ is a semiring.

\textbf{(ii) Morphisms.} A morphism $(\operatorname{CSpec}(R), \mathcal{O}_{\operatorname{CSpec}(R)}) \to (\operatorname{CSpec}(S), \mathcal{O}_{\operatorname{CSpec}(S)})$ is a pair $(f,f^{\#})$, where $f : \operatorname{CSpec}(R) \to \operatorname{CSpec}(S)$ is a continuous map and
\begin{align*}
f^{\#} : \mathcal{O}_{\operatorname{CSpec}(S)} \to f_{\ast}\mathcal{O}_{\operatorname{CSpec}(R)}
\end{align*}
is a morphism of sheaves of semirings.
Here, a \textit{morphism of sheaves of semirings} is understood to be a morphism of sheaves whose components are semiring homomorphisms.

\textbf{(iii) Composition.} Morphisms are composed in the usual way:
if
\begin{align*}
(f,f^\#): (X,\mathcal{O}_X) \to (Y, \mathcal{O}_Y), \qquad (g,g^\#): (Y,\mathcal{O}_Y) \to (Z,\mathcal{O}_Z),
\end{align*}
are morphisms, then their composition is
\begin{align*}
(g,g^\#)\circ(f,f^\#) = (g \circ f,\, (g_{\ast}f^{\#}) \circ g^{\#}),
\end{align*}
where $g_{\ast}f^{\#} :g_{\ast}\mathcal{O}_Y \to g_{\ast}f_{\ast}\mathcal{O}_X$ is the morphism obtained by applying the direct image functor $g_{\ast}$ to $f^{\#}$, and we identify $g_{\ast}f_{\ast} \mathcal{O}_X$ with $(g \circ f)_{\ast} \mathcal{O}_X$ via the natural identification.
The identity morphism is $\operatorname{id}_{(\operatorname{CSpec}(S), \mathcal{O}_{\operatorname{CSpec}(S)})} = ( \operatorname{id}_{\operatorname{CSpec}(S)}, \operatorname{id}_{\mathcal{O}_{\operatorname{CSpec}(S)}})$.

A morphism $(f, f^\#) : (\operatorname{CSpec}(R), \mathcal{O}_{\operatorname{CSpec}(R)}) \to (\operatorname{CSpec}(S), \mathcal{O}_{\operatorname{CSpec}(S)})$ in $\mathsf{CSpec}$ induces a semiring homomorphism on stalks $(f^{\#})_P : \mathcal{O}_{\operatorname{CSpec}(S),f(P)} \to \mathcal{O}_{\operatorname{CSpec}(R),P}$ for each $P \in \operatorname{CSpec}(R)$.
Moreover, there is a canonical semiring homomorphism $\operatorname{ev}_{f(P)} :\mathcal{O}_{\operatorname{CSpec}(S),f(P)} \to S / f(P); a_{f(P)} \mapsto [a]$ (resp.~$ \operatorname{ev}_P : \mathcal{O}_{\operatorname{CSpec}(R),P} \to R / P; b_P \mapsto [b]$).
Here, $a_{f(P)}$ (resp.~$b_P$) denotes the germ at $f(P)$ (resp.~$P$) of the section locally represented by $a \in S$ (resp.~$b \in R$).

\begin{definition}
    \label{dfn2}
The morphism $(f, f^\#)$ is \textit{congruence-reflecting at $P \in \operatorname{CSpec}(R)$} if there exists an injective semiring homomorphism $\overline{(f^{\#})_P} : S / f(P) \hookrightarrow R / P$ such that $\overline{(f^{\#})_P} \circ \operatorname{ev}_{f(P)} =\operatorname{ev}_P \circ (f^{\#})_P$.
The morphism $(f, f^\#)$ is \textit{congruence-reflecting} if it is congruence-reflecting at every $P \in \operatorname{CSpec}(R)$.
\end{definition}

The congruence-reflecting morphisms are closed under composition and contain the identity morphisms, as is readily verified from the definition.

We define $\mathsf{CSpec}_{\mathrm{abstract}}$ to be the subcategory of $\mathsf{CSpec}$ whose objects are
\begin{align*}
\left( \operatorname{CSpec}(A),\mathcal{O}_{\operatorname{CSpec}(A)} \right)
\end{align*}
for $A$ in $\mathsf{Alg}_{\mathrm{abstract}}$ and whose morphisms are the congruence-reflecting morphisms.

\begin{proposition}
    \label{prop2}
For $A, B \in \mathsf{Alg}_{\mathrm{abstract}}$, let $\psi : A \to B$ be a $\boldsymbol{T}$-algebra homomorphism.
Then the continuous map
\begin{align*}
f_\psi : \operatorname{CSpec}(B) \to \operatorname{CSpec}(A); P \mapsto \psi^{-1}(P),
\end{align*}
together with the morphism of sheaves of semirings
\begin{align*}
f^{\#}_\psi :
\mathcal{O}_{\operatorname{CSpec}(A)}
\to
(f_{\psi})_{\ast}\mathcal{O}_{\operatorname{CSpec}(B)},
\end{align*}
defines a morphism
\begin{align*}
\operatorname{CSpec}(\psi) = (f_{\psi}, f^{\#}_{\psi}):
\left(\operatorname{CSpec}(B),\mathcal{O}_{\operatorname{CSpec}(B)}\right)
\to
\left(\operatorname{CSpec}(A),\mathcal{O}_{\operatorname{CSpec}(A)}\right).
\end{align*}
Here, for an open subset $U\subset\operatorname{CSpec}(A)$ and a section $s \in \mathcal{O}_{\operatorname{CSpec}(A)}(U)$, if $s$ is locally represented by $a\in A$ at $P \in U$, then $f^{\#}_{\psi}(U)(s) \in \mathcal{O}_{\operatorname{CSpec}(B)}(f_{\psi}^{-1}(U))$ is locally represented by $\psi(a)\in B$ at every point of $f_{\psi}^{-1}(P)$.
Moreover, $\operatorname{CSpec}$ defines a contravariant functor $\mathsf{Alg}_{\mathrm{abstract}} \to \mathsf{CSpec}_{\mathrm{abstract}}$.
\end{proposition}

\begin{proof}
By definition, $f_{\psi}(P) = \psi^{-1}(P)$ belongs to $\operatorname{CSpec}(A)$ for every $P \in \operatorname{CSpec}(B)$.
Moreover, since for any congruence $C$ on $A$,
\begin{align*}
f_{\psi}^{-1}(V(C))
&=
\{P\in\operatorname{CSpec}(B)\mid \psi^{-1}(P)\supset C\}\\
&=
\{P\in\operatorname{CSpec}(B)\mid P\supset\langle\psi(C)\rangle_B\}\\
&=
V(\langle\psi(C)\rangle_B),
\end{align*}
$f_{\psi}$ is continuous.

Let $s\in\mathcal{O}_{\operatorname{CSpec}(A)}(U)$, and suppose that $s$ is locally represented by $a\in A$ at $P\in U$.
Then there exists a neighborhood $V\subset U$ of $P$ such that $s(Q)=[a]\in A/Q$ for every $Q\in V$.
For every $Q^{\prime} \in f_\psi^{-1}(V)$, $f_{\psi}(Q')=\psi^{-1}(Q')\in V,$ and hence $f_{\psi}^{\#}(U)(s)(Q') = [\psi(a)]\in B/Q^{\prime}$.
Thus $f_{\psi}^{\#}(U)(s)$ is locally represented by $\psi(a)$ at every point of $f_{\psi}^{-1}(V)$, and in particular at every point of $f_{\psi}^{-1}(P)$.
Therefore $f_{\psi}^{\#}(U)(s) \in \mathcal{O}_{\operatorname{CSpec}(B)}(f_{\psi}^{-1}(U))$.

Assume that $s$ is locally represented by two elements $a, a^{\prime}$ of $A$ at $P \in U$.
Then there exists a neighborhood $V \subset U$ of $P$ such that for any $Q \in V, s(Q) = [a] = [a^{\prime}] \in A / Q$.
Therefore $(a,a^{\prime}) \in Q$ for every $Q \in V$.
For any $Q^{\prime} \in f_{\psi}^{-1}(V)$, since $f_{\psi}(Q^{\prime}) = \psi^{-1}(Q^{\prime}) \in V$, $(a, a^{\prime}) \in \psi^{-1}(Q^{\prime})$.
Thus $(\psi(a), \psi(a^{\prime})) \in Q^{\prime}$, which means that $[\psi(a)] = [\psi(a^{\prime})] \in B / Q^{\prime}$.

The maps $f_{\psi}^{\#}(U): \mathcal{O}_{\operatorname{CSpec}(A)}(U) \to \mathcal{O}_{\operatorname{CSpec}(B)}(f_{\psi}^{-1}(U))$ are semiring homomorphisms, since they are defined locally by the semiring homomorphism $\psi$.
They are compatible with restriction maps by construction.
Hence $f_{\psi}^{\#}$ is a morphism of sheaves of semirings.

For $P \in \operatorname{CSpec}(B)$, by \cite[Proposition 2.4.4]{Giansiracusa=Giansiracusa2}, the $\boldsymbol{T}$-algebra homomorphism $\overline{(f_{\psi}
^{\#})_P} : A / \psi^{-1}(P) \to B/P; [a] \mapsto [\psi(a)]$ is well-defined and injective.
Moreover, it satisfies $\overline{(f_{\psi}^{\#})_P} \circ \operatorname{ev}_{\psi^{-1}(P)} = \operatorname{ev}_{P} \circ (f_{\psi}^{\#})_P$.
Thus $\operatorname{CSpec}(\psi)=(f_{\psi}, f_{\psi}^{\#})$ is congruence-reflecting.

It remains to verify functoriality.
For the identity homomorphism $\operatorname{id}_A$ of $A$, $f_{\operatorname{id}_A}=\operatorname{id}_{\operatorname{CSpec}(A)}$ and for any open subset $U \subset \operatorname{CSpec}(A)$, $f_{\operatorname{id}_A}^{\#}(U) = \operatorname{id}_{\mathcal{O}_{\operatorname{CSpec}(A)}(U)}$.
Moreover, if $A \xrightarrow{\psi} B \xrightarrow{\phi} C$ are $\boldsymbol{T}$-algebra homomorphisms, then $f_{\phi\circ\psi} = f_{\psi}\circ f_{\phi}$, since $(\phi\circ\psi)^{-1}(P) = \psi^{-1}(\phi^{-1}(P))$ for every $P\in\operatorname{CSpec}(C)$.
Similarly, both $f_{\phi\circ\psi}^{\#}$ and $((f_{\psi})_{\ast} f_{\phi}^{\#}) \circ f_{\psi}^{\#}$ send a local representative $a\in A$ to the local representative $\phi(\psi(a))\in C$.
Hence $f_{\phi\circ\psi}^{\#} = ((f_{\psi})_{\ast} f_{\phi}^{\#}) \circ f_{\psi}^{\#}$.
Therefore $\operatorname{CSpec}$ is a contravariant functor.
\end{proof}

Let $\mathsf{CSpec}_{\mathrm{Geom}}$ be the full subcategory of $\mathsf{CSpec}_{\mathrm{abstract}}$ whose objects are $(\operatorname{CSpec}(A), \mathcal{O}_{\operatorname{CSpec}(A)})$ for $A$ in $\mathsf{Alg}_{\mathrm{Geom}}$.

\begin{theorem}[Categorical Equivalence]
    \label{thm2}
The contravariant functor $\operatorname{CSpec} : \mathsf{Alg}_{\mathrm{abstract}} \to \mathsf{CSpec}_{\mathrm{abstract}}$ induces an equivalence of categories
\begin{align*}
\mathsf{Alg}_{\mathrm{Geom}}^{\mathrm{op}}
\simeq
\mathsf{CSpec}_{\mathrm{Geom}}.
\end{align*}
\end{theorem}

\begin{proof}
By definition, $\operatorname{CSpec} : \mathsf{Alg}_{\mathrm{Geom}} \to \mathsf{CSpec}_{\mathrm{Geom}}$ is essentially surjective. 

Let $(f,f^{\#}): (\operatorname{CSpec}(B), \mathcal{O}_{\operatorname{CSpec}(B)}) \to (\operatorname{CSpec}(A), \mathcal{O}_{\operatorname{CSpec}(A)})$ be a morphism in $\mathsf{CSpec}_{\mathrm{Geom}}$.
Then by Theorem~\ref{thm1}, the induced homomorphism on global sections $f^{\#}(\operatorname{CSpec}(A)) : \Gamma(\operatorname{CSpec}(A),\mathcal{O}_{\operatorname{CSpec}(A)}) \to \Gamma(\operatorname{CSpec}(B), \mathcal{O}_{\operatorname{CSpec}(B)})$ defines the $\boldsymbol{T}$-algebra homomorphism $\psi = \eta_B^{-1} \circ f^{\#}(\operatorname{CSpec}(A)) \circ \eta_A : A \to B$.
On the other hand, since $(f, f^{\#})$ is congruence-reflecting, for each $P \in \operatorname{CSpec}(B)$, there exists an injective $\boldsymbol{T}$-algebra homomorphism $\overline{(f^{\#})_P}:A/f(P) \hookrightarrow B/P$ such that $\overline{(f^{\#})_P} \circ \operatorname{ev}_{f(P)} =\operatorname{ev}_P \circ (f^{\#})_P$.
Hence, for $a \in A$, $\overline{(f^{\#})_P}([a]) = [\psi(a)] \in B / P$.
Thus $(f^{\#})_P = (f_{\psi}^{\#})_P$ for every $P \in \operatorname{CSpec}(B)$, and so $f^{\#} = f_{\psi}^{\#}$.
Moreover, by the injectivity of $\overline{(f^{\#})_P}$, for $a, a^{\prime} \in A$, $(a, a^{\prime}) \in f(P)$ if and only if $(\psi(a), \psi(a^{\prime})) \in P$.
This implies that $f(P) = \psi^{-1}(P)$, namely, $f = f_{\psi}$.
Thus, $\operatorname{CSpec}$ is full.

If $\psi_1, \psi_2 : A \to B$ are $\boldsymbol{T}$-algebra homomorphisms such that $\operatorname{CSpec}(\psi_1) = \operatorname{CSpec}(\psi_2)$, then $f_{\psi_1}^{\#}(\operatorname{CSpec}(A)) \circ \eta_A = f_{\psi_2}^{\#}(\operatorname{CSpec}(A)) \circ \eta_A$ implies $\eta_B \circ \psi_1 = \eta_B \circ \psi_2$.
Since $\eta_B$ is a $\boldsymbol{T}$-algebra isomorphism by Theorem~\ref{thm1}, $\psi_1 = \psi_2$.
Therefore, $\operatorname{CSpec}$ is faithful.
\end{proof}

\begin{remark}
    \label{rem3}
The equivalence established above is not specific to $\overline{\boldsymbol{T}(\boldsymbol{X})}_n/\boldsymbol{E}(V)$.
By Theorem~\ref{thm1}, the same construction applies to
\begin{align*}
\boldsymbol{T}[\boldsymbol{X}^{\pm}]_n/\boldsymbol{E}(V)
\quad\text{and}\quad
\overline{\boldsymbol{T}[\boldsymbol{X}^{\pm}]}_n/\boldsymbol{E}(V).
\end{align*}
Thus, for each of these three choices of $K$, the category of congruence spectra equipped with the corresponding structure sheaves is equivalent to the opposite of the corresponding category of $\boldsymbol{T}$-algebras.
\end{remark}

\section{The case of finite unions of $\boldsymbol{R}$-rational polyhedral sets}
    \label{section4}

In this section, we give a complete classification of the prime congruences containing $\boldsymbol{E}(V)$ when $V$ is a finite union of nonempty $\boldsymbol{R}$-rational polyhedral sets.
We first prepare the necessary notions concerning $\boldsymbol{R}$-rational polyhedral sets.
We then give a complete classification of these prime congruences.
This provides a geometric description of the prime congruences appearing in the spectrum associated with $V$, and will be applied in the next section to tropical curves with parallel rays.

\subsection{$\boldsymbol{R}$-rational polyhedral sets}
    \label{subsection4.1}

A \textit{polyhedral set} in $\boldsymbol{R}^n$ is the solution set of a system of a finite number of linear inequalities.
The polyhedral set $\{ \boldsymbol{x} \in \boldsymbol{R}^n \,|\, A \boldsymbol{x} \ge \boldsymbol{0} \}$ is called the \textit{recession cone} of the polyhedral set $P := \{ \boldsymbol{x} \in \boldsymbol{R}^n \,|\, A \boldsymbol{x} \ge \boldsymbol{b} \}$ defined by an $l \times n$-matrix $A$ and a vector $\boldsymbol{b} \in \boldsymbol{R}^l$ when $P$ is nonempty.
A \textit{finite polyhedral complex} $X$ is a complex consisting of a finite number of polyhedral sets.
Its \textit{support} $|X|$ is the union of its polyhedral sets.
The \textit{recession fan} of $X$ is the union of all recession cones of polyhedral sets in $X$.
By definition, we can say that it is also the \textit{recession fan} $\operatorname{rec}(|X|)$ of $|X|$.
A finite union of polyhedral sets in $\boldsymbol{R}^n$ is the support of a finite polyhedral complex in $\boldsymbol{R}^n$, and vice versa (cf.~\cite[Proposition~4.1.1(a)]{Mikhalkin=Rau}).

For a nonempty polyhedral set $P$ in $\boldsymbol{R}^n$, its \textit{dimension} $\operatorname{dim}P$ is defined by the dimension of the smallest affine subspace of $\boldsymbol{R}^n$ containing $P$.
If $P$ is empty, we define that it has dimension minus one.
The \textit{dimension} $\operatorname{dim}X$ of a finite polyhedral complex $X$ consisting of polyhedral sets $P_1, \ldots, P_m$ is defined as $\operatorname{max}\{ \operatorname{dim}P_i \,|\, i = 1, \ldots, m\}$.
We say that $|X|$ has dimension $\operatorname{dim}X$, i.e., $\operatorname{dim}(|X|) := \operatorname{dim}X$.
Since $X$ consists of only a finite number of polyhedral sets, $\operatorname{dim}(|X|)$ is independent of the choice of its \textit{polyhedral structures}, i.e., finite polyhedral complexes with $|X|$ as their supports.

A polyhedral set $P$ in $\boldsymbol{R}^n$ is \textit{$\boldsymbol{R}$-rational} if $P$ is of the form $\{ \boldsymbol{x} \in \boldsymbol{R}^n \,|\, A\boldsymbol{x} \ge \boldsymbol{b} \}$ with some $l \times n$-matrix $A$ with entries in $\boldsymbol{Q}$ and some vector $\boldsymbol{b} \in \boldsymbol{R}^l$.

For $x_1, \ldots, x_m \in \boldsymbol{R}^n$, let $\operatorname{Cone}\{x_1, \ldots, x_m \}$ denote the polyhedral cone $\left\{ \sum_{i = 1}^m t_ix_i \,\middle|\, t_i \ge 0 \right\}$ generated by them.
We set $\operatorname{Cone}(\varnothing) := \{ \boldsymbol{0}\}$.

For $A, B \subset \boldsymbol{R}^n$, let $A + B$ be the \textit{Minkowski sum} $\{ x + y \,|\, x \in A, y \in B \}$.

\subsection{Classification of prime congruences}
    \label{subsection4.2}

Let $V$ be a finite union of nonempty $\boldsymbol{R}$-rational polyhedral sets in $\boldsymbol{R}^n$.
Let $d$ be its dimension and $K$ one of $\boldsymbol{T}[\boldsymbol{X}^{\pm}]_n, \overline{\boldsymbol{T}[\boldsymbol{X}^{\pm}]}_n, \overline{\boldsymbol{T}(\boldsymbol{X})}_n$.

By Remark~\ref{rem1} and \cite[Corollaries~3.26 and 4.6]{Nakajima=Song}, the quotient $\boldsymbol{T}$-algebra $K / \boldsymbol{E}(V)$ has Krull dimension $d + 1$.

By \cite[Proposition~2.13 and Proposition~3.9(ii), (iii)]{Joo=Mincheva2}, for any t-admissible matrix $U$, $\boldsymbol{E}(V) \subset P(U)$ on $\boldsymbol{T}[\boldsymbol{X}^{\pm}]_n$ if and only if $\boldsymbol{E}(V) \subset P(U)$ on $\overline{\boldsymbol{T}[\boldsymbol{X}^{\pm}]}_n$ if and only if $\boldsymbol{E}(V) \subset P(U)$ on $\overline{\boldsymbol{T}(\boldsymbol{X})}_n$.

By \cite[Proposition~2.13]{Joo=Mincheva2} and \cite[Lemmas~3.11 and 3.12]{Nakajima=Song}, for a prime congruence $P$ on $K$, if $\boldsymbol{E}(V) \subset P$, then a t-admissible matrix $U$ such that $P = P(U)$ can be chosen to be one of the following forms:
\begin{align*}
\begin{pmatrix} 0 & \cdots & 0 \end{pmatrix}, \quad U(x_1, \ldots, x_l; k) := \begin{pmatrix} 0 & & x_1 \\ & \vdots & \\ 0 & & x_{k-1} \\ 1 && x_k \\ 0 & &x_{k + 1} \\ & \vdots & \\ 0 && x_l \end{pmatrix}, \quad Z(y_1, \ldots, y_m) := \begin{pmatrix} 0 & & y_1 \\ & \vdots & \\ 0 & & y_m \end{pmatrix},
\end{align*}
where $1 \le l, m \le d + 1, 1 \le k \le l$, $x_1, \ldots, x_l, y_1, \ldots, y_m \in \boldsymbol{R}^n$ and $x_1, \ldots, x_{k - 1}, x_{k + 1}, \ldots, x_l$ are linearly independent, as are $y_1, \ldots, y_m$.

By \cite[the proof of Theorem~1.1]{JuAe6}, there exists $f \in \overline{\boldsymbol{T}(\boldsymbol{X})}_n \setminus \{ -\infty \}$ such that $\boldsymbol{E}(V) = \langle (f, 0) \rangle_{\overline{\boldsymbol{T}(\boldsymbol{X})}_n}$.
Choose $f_1, f_2 \in \boldsymbol{T}[\boldsymbol{X}^{\pm}]_n \setminus \{ -\infty \}$ such that $f = \iota(\pi(f_1)) \odot \iota(\pi(f_2))^{\odot (-1)}$, where $\pi$ and $\iota$ are as in Remark~\ref{rem1}.
Clearly $(f_1, f_2) \in \boldsymbol{E}(V)$ on $\boldsymbol{T}[\boldsymbol{X}^{\pm}]_n$ and $V = \boldsymbol{V}(\boldsymbol{E}(V)) = \boldsymbol{V}(\{(f,0)\}) = \boldsymbol{V}(\{(f_1, f_2)\})$.

\textbf{Case $0$. $U = \begin{pmatrix} 0 & \cdots & 0 \end{pmatrix}$.}
Since $V \not= \varnothing$, for $x \in V$, $\boldsymbol{E}(V) \subset \boldsymbol{E}(\{x\}) = P(U(x;1)) \subset P(\begin{pmatrix} 0 & \cdots & 0 \end{pmatrix})$ hold.

\textbf{Case $1$. $U = Z(y_1, \ldots, y_m)$.}
By \cite[the proof of Proposition~3.22 and Lemma~3.26]{Nakajima=Song}, $\boldsymbol{E}(V) \subset P(U)$ if and only if $\varepsilon_2,\ldots,\varepsilon_m>0$ can be chosen successively, where $\varepsilon_i$ is chosen after $\varepsilon_2,\ldots,\varepsilon_{i-1}$, so that $\operatorname{rec}(V) \supset \operatorname{Cone}\left\{ y_1, \ldots, y_1 + \sum_{i = 2}^m \varepsilon_i y_i \right\}$.
In particular, $m \le d$.

\textbf{Case $2$. $U = U(x_1, \ldots, x_l;k)$.}

\textbf{Subcase $2$-$1$. $k = 1$.}
By \cite[the proofs of Propositions~3.18 and 3.19]{Nakajima=Song}, $\boldsymbol{E}(V) \subset P(U)$ if and only if $\varepsilon_2,\ldots,\varepsilon_l>0$ can be chosen successively, where $\varepsilon_i$ is chosen after $\varepsilon_2,\ldots,\varepsilon_{i-1}$, so that $V \supset \left\{ x_1 + \sum_{i = 2}^l \delta_i x_i \,\middle|\, 0 \le \delta_i \le \varepsilon_i \right\}$.

\textbf{Subcase $2$-$2$. $2 \le k \le l$.}
Since $k \ge 2$ and $P(U) \subset P(U(k-1))$ and $U(k-1) = Z(x_1, \ldots, x_{k-1})$, if $\boldsymbol{E}(V) \subset P(U)$, then, by Case~$1$, there exist $\varepsilon_2, \ldots, \varepsilon_{k - 1} > 0$ such that $\operatorname{rec}(V) \supset \operatorname{Cone}\left\{ x_1, \ldots, x_1 + \sum_{i = 2}^{k - 1} \varepsilon_i x_i \right\} =: L$.
By \cite[Proposition~2.10(i)]{Joo=Mincheva}, $f_1$ (resp.~$f_2$) has a term $m_{f_1}$ (resp.~$m_{f_2}$) such that $(f_1, m_{f_1}) \in P(U)$ (resp.~$(f_2, m_{f_2}) \in P(U)$).
By the definition of $P(U)$ and \cite[Remark~2.6]{Nakajima=Song}, there exist $z \in \{ x_k \} + L$ and $\varepsilon_{k + 1}, \ldots, \varepsilon_l > 0$ such that $f_1 = m_{f_1} = m_{f_2} = f_2$ on $\left\{ z + \sum_{i = k + 1}^l \delta_i x_i \,\middle|\, 0 \le \delta_i \le \varepsilon_i \right\} + L =: W$.
Hence $W \subset V$.
Note that for $i = k+1, \ldots, l$, $\varepsilon_i$ is chosen after $\varepsilon_2, \ldots, \varepsilon_{k - 1}, \varepsilon_{k + 1}, \ldots, \varepsilon_{i - 1}$.

Assume that there exist $\varepsilon_2, \ldots, \varepsilon_{k - 1}, \varepsilon_{k + 1}, \ldots, \varepsilon_l > 0$ and $z \in \{ x_k \} + \operatorname{Cone}\left\{ x_1, \ldots, x_1 + \sum_{i = 2}^{k - 1} \varepsilon_i x_i \right\}$ such that $V \supset \left\{ z + \sum_{i = k + 1}^l \delta_i x_i \,\middle|\, 0 \le \delta_i \le \varepsilon_i \right\} + \operatorname{Cone}\left\{ x_1, \ldots, x_1 + \sum_{i = 2}^{k-1} \varepsilon_i x_i \right\} =: W$.
Since $f_1$ and $f_2$ are piecewise linear, there exist $w \in W$ and $0 < \delta_i < \varepsilon_i$ for each $i = 2, \ldots, k - 1, k + 1, \ldots, l$ and terms $m_{f_1}$ of $f_1$ and $m_{f_2}$ of $f_2$, respectively, such that $f_1 = m_{f_1}$ and $f_2 = m_{f_2}$ on $\left\{ w + \sum_{i = k + 1}^l \delta_i^{\prime} x_i \,\middle|\, 0 \le \delta_i^{\prime} \le \delta_i \right\} + \operatorname{Cone}\left\{ x_1, \ldots, x_1 + \sum_{i = 2}^{k-1} \delta_i x_i \right\} =: W^{\prime}$ and $W^{\prime} \subset W$.
Then, since $W^{\prime} \subset W \subset V$, on $W^{\prime}$, $f_1$ and $f_2$ coincide.
Thus $(m_{f_1}, m_{f_2}) \in P(U)$, and hence $(f_1, f_2) \in P(U)$ by \cite[Remark~2.6]{Nakajima=Song}.
Consequently, $(f,0) \in P(U)$, which means that $\boldsymbol{E}(V) = \langle (f,0) \rangle_{\overline{\boldsymbol{T}(\boldsymbol{X})}_n} \subset P(U)$.

\smallskip

The preceding cases give a complete geometric criterion for the inclusion $\boldsymbol{E}(V) \subset P(U)$.
We summarize the result in the following theorem:

\begin{theorem}[Complete Classification]
    \label{thm3}
Let $V \subset \boldsymbol{R}^n$ be a finite union of nonempty $\boldsymbol{R}$-rational polyhedral sets in $\boldsymbol{R}^n$, and let $d = \operatorname{dim}(V)$.
For a prime congruence $P$ on one of $\boldsymbol{T}[\boldsymbol{X}^{\pm}]_n, \overline{\boldsymbol{T}[\boldsymbol{X}^{\pm}]}_n, \overline{\boldsymbol{T}(\boldsymbol{X})}_n$, the inclusion $\boldsymbol{E}(V) \subset P$ holds if and only if 
$P = P(U)$ for a $t$-admissible matrix $U$ of one of the following forms, with the corresponding conditions:

Case $0$: $U = \begin{pmatrix} 0 & \cdots &0 \end{pmatrix}$.

Case $1$: $U = Z(y_1, \ldots, y_m)$, where $m \le d$, $y_1, \ldots, y_m \in \boldsymbol{R}^n$ are linearly independent, and there exist $\varepsilon_2, \ldots, \varepsilon_m > 0$ such that 
\begin{align*}
\operatorname{Cone}\left\{ y_1, \ldots, y_1 + \sum_{i = 2}^m \varepsilon_i y_i \right\} \subset \operatorname{rec}(V).
\end{align*}

Case $2$: $U = U(x_1, \ldots, x_l;k)$, where $1 \le k \le l \le d + 1$, $x_1, \ldots, x_l \in \boldsymbol{R}^n$, $x_1, \ldots, x_{k-1}, x_{k+1}, \ldots, x_l$ are linearly independent, and there exist $\varepsilon_2, \ldots, \varepsilon_{k - 1}, \varepsilon_{k + 1}, \ldots, \varepsilon_l > 0$ and $z \in \{ x_k \} + \operatorname{Cone}\left\{ x_1, \ldots, x_1 + \sum_{i = 2}^{k - 1} \varepsilon_i x_i \right\}$ such that
\begin{align*}
\left\{ z + \sum_{i = k + 1}^l \delta_i x_i \,\middle|\, 0 \le \delta_i \le \varepsilon_i \right\} + \operatorname{Cone}\left\{ x_1, \ldots, x_1 + \sum_{i = 2}^{k - 1} \varepsilon_i x_i \right\} \subset V.
\end{align*}
\end{theorem}

Clearly $\operatorname{CSpec}(\overline{\boldsymbol{T}(\boldsymbol{X})}_n / \boldsymbol{E}(\varnothing)) = \operatorname{CSpec}(\{-\infty\}) =\varnothing$.
Together with this observation, Theorem~\ref{thm3} essentially gives a complete classification of the elements of $\operatorname{CSpec}(\overline{\boldsymbol{T}(\boldsymbol{X})}_n / \boldsymbol{E}(V))$ for any finite union $V$ of $\boldsymbol{R}$-rational polyhedral sets in $\boldsymbol{R}^n$.

\begin{remark}
    \label{rem4}
Our classification is closely related to the geometric characterization of prime congruences obtained by Friedenberg and Mincheva~\cite[Theorem~1.1]{Friedenberg=Mincheva}.
They show that, for a congruence $E$ with a finite tropical basis on a toric semiring $\mathcal{S}[\mathcal{M}]$, a prime congruence $P$ contains $E$ if and only if $P=P_{\mathcal C_\bullet}$ for a flag of cones $\mathcal C_\bullet$ contained in the corresponding enhanced variety $\widetilde{V}(E)$.
In the setting considered here, where $E=\boldsymbol{E}(V)$ and $V$ is a finite union of nonempty $\boldsymbol{R}$-rational polyhedral sets, Theorem~\ref{thm3} makes this geometric criterion explicit in terms of the classification of prime congruences by $t$-admissible matrices.
In particular, the condition that the defining matrix of $P$ has its rows in the enhanced variety is translated into the existence of explicit polyhedral subsets of $V$ and $\operatorname{rec}(V)$ according to the three cases in Theorem~\ref{thm3}.
\end{remark}

\section{The case of tropical curves with parallel rays}
    \label{section5}

In this section, we apply the classification obtained in the previous section to the rational function semifield of a tropical curve with parallel rays.
We first set up the necessary framework by recalling tropical curves with parallel rays and their rational functions.
We then classify the prime congruences on their rational function semifields and show that their spectra encode the underlying tropical curves, the parallelism of their rays, and their intrinsic metrics.

\subsection{Tropical curves with parallel rays and rational functions}
    \label{subsection5.1}

In this paper, a \textit{graph} is an unweighted, undirected, finite connected nonempty multigraph that may have loops.
For a graph $G$, the set of vertices is denoted by $V(G)$ and the set of edges by $E(G)$.
A vertex $v$ of $G$ is a \textit{leaf end} if $v$ is incident to only one edge and this edge is not a loop.
A \textit{leaf edge} is an edge of $G$ incident to a leaf end.

An \textit{(abstract) tropical curve} is the underlying topological space of the pair $(G, l)$ of a graph $G$ and a function $l: E(G) \to {\boldsymbol{R}}_{>0} \cup \{\infty\}$, where $l$ can take the value $\infty$ only on leaf edges, together with the identification of each edge $e$ of $G$ with the closed interval $[0, l(e)]$.
The interval $[0, \infty]$ is the one-point compactification of $[0, \infty)$, regarded as an extended metric space where the distance between $\infty$ and any other point is infinite.
When a tropical curve $\varGamma$ is obtained from $(G, l)$, the pair $(G, l)$ is called a \textit{model} for $\varGamma$.
For a point $x \in \varGamma$, if $x$ corresponds to $\infty$, it is called a \textit{point at infinity}; otherwise, $x$ is a \textit{finite point}.
We write the set of all points at infinity of $\varGamma$ as $\varGamma_{\infty}$.
The \textit{genus} of $\varGamma$ is its first Betti number, which coincides with $\# E(G) - \#V(G) + 1$ for any model $(G,l)$ for $\varGamma$.
For a finite point $x \in \varGamma \setminus \varGamma_{\infty}$, the \textit{valency} $\operatorname{val}(x)$ is the number of connected components of $U \setminus \{ x \}$ for any sufficiently small connected neighborhood $U$ of $x$; if $x \in \varGamma_{\infty}$, we define $\operatorname{val}(x) := 1$.
A \textit{ray} of $\varGamma$ is the closed subset of $\varGamma$ corresponding to an edge of infinite length in some model, which necessarily contains a point at infinity.

A continuous map $f : \varGamma \to \boldsymbol{R} \cup \{ \pm \infty \}$ is a \textit{rational function} on $\varGamma$ if $f$ is identically $-\infty$ or a piecewise affine function with integer slopes with a finite number of pieces, taking the values $\pm \infty$ only at points at infinity.
Let $\operatorname{Rat}(\varGamma)$ denote the set of all rational functions on $\varGamma$.
With the pointwise operations $(f \oplus g)(x) := \max\{f(x), g(x)\}$ and $(f \odot g)(x) := f(x) + g(x)$ (continuously extended to $\varGamma_{\infty}$), $\operatorname{Rat}(\varGamma)$ becomes a semifield over $\boldsymbol{T}$ and a $\boldsymbol{T}$-algebra.

An \textit{(abstract) tropical curve with parallel rays} $\varGamma_{\operatorname{par}}$ is an (abstract) tropical curve $\varGamma$ whose rays have an equivalence relation $\sim$ such that for rays $L_1$ and $L_2$ of $\varGamma$, if one ray is contained in another, then $L_1 \sim L_2$ holds.
A \textit{ray} of $\varGamma_{\operatorname{par}}$ is a ray of $\varGamma$.
Two rays of $\varGamma_{\operatorname{par}}$ are \textit{parallel} if they are equivalent under $\sim$, otherwise, these are \textit{nonparallel}.
The \textit{genus} of $\varGamma_{\operatorname{par}}$ is that of $\varGamma$ and \textit{points} (resp.~\textit{finite points}, \textit{points at infinity}) of $\varGamma_{\operatorname{par}}$ are those of $\varGamma$.
We write the set of points at infinity of $\varGamma_{\operatorname{par}}$ as $\varGamma_{\operatorname{par}, \infty}$.
The \textit{valency} of a point $x$ of $\varGamma_{\operatorname{par}}$ is $\operatorname{val}(x)$.  

A function $f : \varGamma \to \boldsymbol{R} \cup \{ \pm \infty \}$ is a \textit{rational function} on $\varGamma_{\operatorname{par}}$ if $f \in \operatorname{Rat}(\varGamma)$ and $f$ has the same slope at infinity along any pair of parallel rays of $\varGamma_{\operatorname{par}}$, or $f \equiv - \infty$.
Let $\operatorname{Rat}(\varGamma_{\operatorname{par}})$ denote the set of all rational functions on $\varGamma_{\operatorname{par}}$, which is a finitely generated subsemifield over $\boldsymbol{T}$ of $\operatorname{Rat}(\varGamma)$ by \cite[Proposition~4.4]{JuAe7}.
We call $\operatorname{Rat}(\varGamma_{\operatorname{par}})$ the \textit{rational function semifield} of $\varGamma_{\operatorname{par}}$.

\subsection{Classification of prime congruences}
    \label{subsection5.2}

We use the semifield $R_n := ((\boldsymbol{R} \times \boldsymbol{Z}^n) \cup \{ -\infty \}, \boxplus, \boxdot)$ over $\boldsymbol{T}$ introduced in \cite[Subsection~4.6]{JuAe7}, where, for $(a, (i_1, \ldots, i_n)), (b, (j_1, \ldots, j_n)) \in \boldsymbol{R} \times \boldsymbol{Z}^n, \boldsymbol{x} \in (\boldsymbol{R} \times \boldsymbol{Z}^n) \cup \{ -\infty \}$,
\begin{align*}
(a, (i_1, \ldots, i_n)) \boxplus (b, (j_1, \ldots, j_n)) &:= \begin{cases}
(a, (i_1, \ldots, i_n)) \quad \text{if }a > b,\\
(a, (\operatorname{max}\{ i_1, j_1 \}, \ldots, \operatorname{max}\{ i_n, j_n \})) \quad \text{if }a = b,\\
(b, (j_1, \ldots, j_n)) \quad \text{if }a < b,
\end{cases}\\
\boldsymbol{x} \boxplus (-\infty) &:= -\infty \boxplus \boldsymbol{x} := \boldsymbol{x},\\
(a, (i_1, \ldots, i_n)) \boxdot (b, (j_1, \ldots, j_n)) &:= (a + b, (i_1 + j_1, \ldots, i_n + j_n)),\\
\boldsymbol{x} \boxdot (-\infty) &:= -\infty \boxdot \boldsymbol{x} := -\infty.
\end{align*}

We set $R_0 := \boldsymbol{T}$.

\begin{remark}
    \label{rem5}
Let $K_n$ be $\boldsymbol{T}[\boldsymbol{X}^{\pm}]_n, \overline{\boldsymbol{T}[\boldsymbol{X}^{\pm}]}_n$ or $\overline{\boldsymbol{T}(\boldsymbol{X})}_n$.
For any $n \ge 1$, the following isomorphisms hold:
\begin{itemize}
    \item $K_n / P (\begin{pmatrix} 0 & \cdots & 0 \end{pmatrix} ) \cong \boldsymbol{B}$,
    \item $K_n / P (U(x_1;1)) \cong \boldsymbol{T}$, where $x_1 \in \boldsymbol{R}^n$,
    \item $K_n / P (U(x_1,x_2;1)) \cong K_1 / P (U(0,1;1)) \cong R_1$, where $x_1 \in \boldsymbol{R}^n$ and $x_2 \in \boldsymbol{Z}^n \setminus \{ \boldsymbol{0}\}$,
    \item $K_n / P (Z(y_1)) \cong K_1 / P (Z(1))$, where $y_1 \in \boldsymbol{Z}^n \setminus \{ \boldsymbol{0}\}$, and
    \item $K_n / P (U(x_1,x_2;2)) \cong K_1 / P (U(1,0;2))$, where $x_1 \in \boldsymbol{Z}^n \setminus \{ \boldsymbol{0}\}$ and $x_2 \in \boldsymbol{R}^n$.
\end{itemize}
Note that all of the above quotients are semifields, and quotients corresponding to distinct types are not isomorphic to each other.

Moreover, by the definition of $P(U)$, the following equivalences hold:
\begin{itemize}
    \item for any $x_1, x_1^{\prime} \in \boldsymbol{R}^n$, $x_1 = x_1^{\prime}$ if and only if $P (U(x_1;1)) = P (U(x_1^{\prime};1))$,
    \item for any $x_1, x_2, x_1^{\prime}, x_2^{\prime} \in \boldsymbol{R}^n$ such that $x_2, x_2^{\prime} \not= \boldsymbol{0}$, $x_1 = x_1^{\prime}$ and there exists $\delta > 0$ such that $x_2 = \delta x_2^{\prime}$ if and only if $P (U(x_1,x_2;1)) = P (U(x_1^{\prime},x_2^{\prime};1))$,
    \item for any $y_1, y_1^{\prime} \in \boldsymbol{R}^n \setminus \{ \boldsymbol{0} \}$, there exists $\delta > 0$ such that $y_1 = \delta y_1^{\prime}$ if and only if $P (Z(y_1)) = P (Z(y_1^{\prime}))$,
    \item for any $x_1, x_2, x_1^{\prime}, x_2^{\prime} \in \boldsymbol{R}^n$ such that $x_1, x_1^{\prime} \not= \boldsymbol{0}$, there exist $\delta > 0$ and $\lambda \in \boldsymbol{R}$ such that $x_1 = \delta x_1^{\prime}$ and $x_2 = x_2^{\prime} + \lambda x_1^{\prime}$ if and only if $P(U(x_1,x_2;2)) = P(U(x_1^{\prime},x_2^{\prime};2))$,
\end{itemize}
\end{remark}

Let $\varGamma_{\operatorname{par}}$ be a tropical curve with parallel rays and $K := \operatorname{Rat}(\varGamma_{\operatorname{par}})$.

\begin{definition}
    \label{dfn3}
A prime congruence $P \in \operatorname{CSpec}(K)$ is of \textit{Type $0$} (resp.~\textit{Type F-$1$}, \textit{Type F-$2$}, \textit{Type I-$1$}, and \textit{Type I-$2$}) if $K/P$ is isomorphic to $\boldsymbol{B}$ (resp.~$\boldsymbol{T}, R_1, \boldsymbol{T}[Y^{\pm}]/P(Z(1))$, and $\boldsymbol{T}[Y^{\pm}]/P(U(1,0;2))$).
\end{definition}

\begin{example}[cf.~{\cite[Proposition~4.41]{JuAe7}}]
    \label{ex2}
Let $x \in \varGamma_{\operatorname{par}, \infty}$ and $\sim_x$ the relation on $K$ defined by
\begin{align*}
f \sim_x g \quad \Longleftrightarrow \quad \exists U : \text{a neighborhood of }x \text{ such that } f|_U = g|_U.
\end{align*}
Then $P_x := \{ (f, g) \in K^2 \,|\, f \sim_x g\}$ is a prime congruence on $K$ of Type I-$2$.
In fact, a $\boldsymbol{T}$-algebra isomorphism is given as follows.
Let $f_1, \ldots, f_n \in K \setminus \{ -\infty \}$ be generators of $K$ as a semifield over $\boldsymbol{T}$.
Let $y$ be a finite point on a ray containing $x$ of $\varGamma_{\operatorname{par}}$ such that on the segment $[y, x] \subset \varGamma_{\operatorname{par}}$, all of $f_1, \ldots, f_n$ have constant slopes $s_1, \ldots, s_n$, respectively.
The correspondence $f_i \mapsto [f_i(y) \odot Y^{\odot s_i}]$, where $[f_i(y) \odot Y^{\odot s_i}]$ denotes the equivalence class of $f_i(y) \odot Y^{\odot s_i} \in \boldsymbol{T}[Y^{\pm}]$ in $\boldsymbol{T}[Y^{\pm}] / P (U(1,0;2))$, defines a $\boldsymbol{T}$-algebra homomorphism $\psi : K \to \boldsymbol{T}[Y^{\pm}] / P (U(1,0;2))$.
Since $f_1, \ldots, f_n \not= -\infty$ generates $K$, the greatest common divisor of $s_1, \ldots, s_n$ is one, and so $\psi$ is surjective.
Then we can directly check that $\operatorname{Ker}(\psi) = P_x$.
Note that this is independent of the choice of $f_1, \ldots, f_n$ and $y$.
\end{example}

\begin{theorem}[Complete Classification]
    \label{thm4}
Every prime congruence $P \in \operatorname{CSpec}(K)$ is uniquely classified as one of Types $0$, F-$1$, F-$2$, I-$1$, and I-$2$.
\end{theorem}

\begin{proof}
This is an immediate consequence of \cite[Proposition~4.6]{JuAe7} and Theorem~\ref{thm3}.
\end{proof}

\begin{corollary}
    \label{cor3}
Each prime congruence on $K$ of Type F-$2$ is contained in a unique prime congruence on $K$ of Type F-$1$.
Moreover, each prime congruence on $K$ of Type I-$2$ is contained in a unique prime congruence on $K$ of Type I-$1$.
\end{corollary}

\begin{definition}
    \label{dfn4}
We call the prime congruences on $K$ of Type F-$1$ (resp.~Type I-$2$) \textit{finite points} (resp.~\textit{points at infinity}) of $\operatorname{CSpec}(K)$.
The set $X \subset \operatorname{CSpec}(K)$ consisting of all finite points and points at infinity is called the \textit{geometric point set} of $\operatorname{CSpec}(K)$.
Each element of $X$ is called a \textit{geometric point} of $\operatorname{CSpec}(K)$.    
\end{definition}

\subsection{Geometric points}
    \label{subsection5.3}

We regard $X$ as a subspace of $\operatorname{CSpec}(K)$ with respect to the Zariski topology.

\begin{theorem}[Geometric Points]
    \label{thm5}
The map $\varphi : \varGamma_{\operatorname{par}} \to X$ given by the correspondence
\begin{align*}
\varGamma_{\operatorname{par}} \setminus \varGamma_{\operatorname{par}, \infty} \ni &x \mapsto \operatorname{Ker}(K \to \boldsymbol{T}; f \mapsto f(x)) \in X,\\
\varGamma_{\operatorname{par}, \infty} \ni &x \mapsto P_x \in X
\end{align*}
is a homeomorphism.
\end{theorem}

\begin{proof}
Since the points at infinity, which form closed subsets, in $\varGamma_{\operatorname{par}}$ and $X$ correspond via $\varphi$, by \cite[Proposition~4.6]{JuAe7} and Remark~\ref{rem2}, $\varphi$ is a homeomorphism.
\end{proof}

\begin{corollary}[Encoding of Parallel Rays]
    \label{cor4}
Two points $P$ and $Q$ at infinity of $\operatorname{CSpec}(K)$ are contained in the same prime congruence on $K$ of Type I-$1$ if and only if $\varphi^{-1}(P)$ and $\varphi^{-1}(Q)$ are contained in parallel rays of $\varGamma_{\operatorname{par}}$.
\end{corollary}

\begin{corollary}[Valency of Finite Points]
    \label{cor5}
Each finite point $P$ of $\operatorname{CSpec}(K)$ contains exactly $n$ prime congruences on $K$ of Type F-$2$, where $n$ is the valency of $\varphi^{-1}(P)$ in $\varGamma_{\operatorname{par}}$.
Moreover, if $n \ge 1$, then the quotient of $K$ by the intersection of these prime congruences on $K$ of Type F-$2$ is isomorphic to $R_n$.
\end{corollary}

\begin{proof}
The first statement follows directly from the proof of Theorem~\ref{thm4}.
For the second statement, let $Q_1, \ldots, Q_n$ be the distinct prime congruences on $K$ of Type F-$2$ included by $P$.
By \cite[Propositions~4.41 and 4.42]{JuAe7}, it is enough to check that $\bigcap_{i = 1}^n Q_i = \{ (f, g) \in K^2 \,|\, \exists U: \text{a neighborhood of }\varphi^{-1}(P) \text{ such that } f|_U = g|_U \}$, and it is clear by the proof of Theorem~\ref{thm4}.
\end{proof}

Note that if $n = 0$, then $K / P$ is isomorphic to $R_0 = \boldsymbol{T}$.

\begin{definition}
    \label{dfn5}
For a geometric point $P$ of $\operatorname{CSpec}(K)$, we call the valency of $\varphi^{-1}(P)$ the \textit{valency} of $P$.
If the valency of $P$ is different from two, then we call $P$ an \textit{intrinsic vertex} of $X$ (or $\operatorname{CSpec}(K)$).
\end{definition}

Let $v$ be the number of intrinsic vertices of $\operatorname{CSpec}(K)$ and $e$ half the total sum of the valencies of the intrinsic vertices of $\operatorname{CSpec}(K)$.
Then, the genus of $\varGamma_{\operatorname{par}}$ is canonically determined by the algebraic data:

\begin{corollary}[Genus]
    \label{cor6}
The genus of $\varGamma_{\operatorname{par}}$ coincides with $e - v + 1$.
\end{corollary}

We next show that the algebraic structure of $K$ also determines the intrinsic metric on $X$.

\begin{theorem}[Intrinsic Metric]
    \label{thm6}
There exists an intrinsic metric on $X$ determined purely by $K$.
%There exists an intrinsic metric on $X$ defined purely by the algebraic data of $K$.
This metric is independent of the choice of generators of $K$.
Moreover, under this metric, $\varphi$ becomes an isometry that preserves the parallelism of rays, i.e., an isomorphism of tropical curves with parallel rays.
\end{theorem}

\begin{proof}
For any $x \in \varGamma_{\operatorname{par}} \setminus \varGamma_{\operatorname{par}, \infty}$, let $d_x$ be the distance function $\operatorname{dist}(x, \cdot)$ on $\varGamma_{\operatorname{par}}$.
It is a rational function on $\varGamma_{\operatorname{par}}$, namely, $d_x \in K$.
For any finite point $P \in X$ and any point $Q \in X$, define their distance by $d_{\varphi^{-1}(P)}(\varphi^{-1}(Q))$.
This is well defined by Theorem~\ref{thm5}.
Then $\varphi$ is an isometry, and preserves the parallelism of rays by Corollary~\ref{cor4}.
By \cite[Proposition~4.6]{JuAe7}, this intrinsic metric is independent of the choice of generators of $K$.
\end{proof}

Consequently, the topology induced by this metric agrees with the topology inherited from the congruence spectrum.

\begin{theorem}[Coincidence of Topologies]
    \label{thm7}
The restriction of the Zariski topology on $\operatorname{CSpec}(K)$ to the geometric point set $X$ naturally coincides with the metric topology induced by the intrinsic metric defined in Theorem~\ref{thm6}.
\end{theorem}

\begin{proof}
This is an immediate consequence of Theorems~\ref{thm5} and \ref{thm6}.
\end{proof}

\bibliographystyle{plain}
\bibliography{references}
\end{document}